\documentclass[12pt]{amsart}
\usepackage[english]{babel}
\usepackage[margin=1in]{geometry}
\usepackage{amsmath, amssymb}
\usepackage{graphicx}
\usepackage[colorlinks=true, allcolors=blue]{hyperref}
\usepackage{dutchcal} 
\usepackage{comment}
\usepackage{thm-restate}

\theoremstyle{definition} 
\newtheorem{thm}{Theorem}[section]
\newtheorem{example}{Example}[section]
\newtheorem{prop}[thm]{Proposition} 
\newtheorem{lem}[thm]{Lemma} 
\newtheorem{remark}[thm]{Remark} 
\newtheorem{defn}[thm]{Definition} 
\newtheorem{Cor}[thm]{Corollary}
\newtheorem{question}[thm]{Question}

\newcommand{\cu}{\mathcal {U}}

\newcommand{\zz}{{\mathbb Z}}

\newcommand{\rr}{{\mathbb R}}

\DeclareMathOperator{\spn}{span}

\DeclareMathOperator{\Cone}{Cone}

\newcommand{\gs}{\sigma} 
 
\newcommand{\gt}{\tau}

\newcommand{\gG}{\Gamma}

\DeclareMathOperator{\Lk}{Lk} 
\DeclareMathOperator{\gdim}{gdim} 
\DeclareMathOperator{\St}{St}

\makeatletter
\def\smallunderbrace#1{\mathop{\vtop{\m@th\ialign{##\crcr
   $\hfil\displaystyle{#1}\hfil$\crcr
   \noalign{\kern3\p@\nointerlineskip}%
   \tiny\upbracefill\crcr\noalign{\kern3\p@}}}}\limits}
\makeatother

\title{Embedding complexes into pseudomanifolds}
\author{Kasia Jankiewicz}
\author{Kevin Schreve}

\address{Department of Mathematics, University of British Columbia, Vancouver, BC, Canada V6T 1Z2}
\email{kasia@math.ubc.ca}

\address{Department of Mathematics, Louisiana State University, Baton Rouge, LA~70806}
\email{kschreve@lsu.edu}


\begin{document}
\begin{abstract}
   We show that for $d\geq 2$ every finite $d$-dimensional simplicial complex is a deformation retract of a $(2d-1)$-dimensional pseudomanifold with boundary. Moreover, it embeds as a retract in a closed $(2d-1)$-dimensional pseudomanifold.
\end{abstract}

\maketitle

\section{Introduction}
By general position, every finite $d$-dimensional simplicial complex embeds into $\rr^{2d+1}$. 
This dimension is known to be sharp, for instance van Kampen showed that the $d$-fold join of $3$ points does not embed into $\rr^{2d}$ \cite{vanKampen33}. 
On the other hand, Stallings showed that every finite $d$-dimensional CW-complex is homotopy equivalent to a complex which embeds into $\rr^{2d}$  \cite{Stallings65, DranishnikovRepovs93}. 
By taking a regular neighborhood of the image, this shows that every $d$-dimensional CW-complex is homotopy equivalent to a $2d$-dimensional manifold with boundary. 

In this paper, we are interested in thickenings of simplicial complexes into pseudomanifolds, possibly  with boundary. 
There are many different definitions of pseudomanifolds in the literature. 
The pseudomanifolds we construct are simplicial complexes with the links of vertices PL-homeomorphic to $(d-1)$-dimensional manifolds (possibly with boundary). We call such a complex a \emph{pseudomanifold with isolated singularities}. A  pseudomanifold with isolated singularities is \emph{closed} if all links of vertices are PL-homeomorphic to closed manifolds. 
Our analogue of Stallings theorem in this setting is:

\begin{restatable}{thm}{main}
\label{t:main}
For $d \ge 2$, every finite $d$-dimensional simplicial complex $X$ is a deformation retract of an orientable $(2d-1)$-dimensional pseudomanifold with boundary $P(X)$ with isolated singularities. 
\end{restatable}

More concretely, $P(X)$ is obtained from a certain $(2d-1)$-dimensional manifold with boundary by coning off disjoint codimension-zero submanifolds of the boundary. Of course, the theorem does not hold if $d = 1$.

Via simplicial approximation, every finite $d$-dimensional CW-complex is homotopy equivalent to a finite $d$-dimensional simplicial complex \cite[Thm 2C.5]{Hatcher02}, so we obtain the following.

\begin{Cor}
    Every finite $d$-dimensional CW-complex $X$ is homotopy equivalent to an orientable $(2d-1)$-dimensional pseudomanifold with boundary and with isolated singularities.
\end{Cor}

We are also interested in embedding simplicial complexes into closed pseudomanifolds. 
Combining Theorem \ref{t:main} with a general version of Davis's reflection group trick implies

\begin{restatable}{thm}{mainclosed}
\label{t:main2}
Every finite $d$-dimensional simplicial complex $X$ embeds as a retract into a closed $(2d-1)$-dimensional pseudomanifold $Q(X)$ with isolated singularities. 
If $X$ is aspherical, then $Q(X)$ can be chosen to be aspherical. 
\end{restatable}

The idea of the proof of Theorem \ref{t:main} is simpler than Stallings in the sense that we do not have to homotope the simplicial complex $L$. 
Instead, we consider the punctured simplicial complex $L - L^{(0)}$, where $L^{(0)}$ is the vertex set of $L$, which deformation retracts onto an $(n-1)$-dimensional subcomplex $K$ spanned by barycenters of simplices of dimension $> 0$.
Since $K$ is $(n-1)$-dimensional, by general position it embeds into $\rr^{2n-1}$. 
We construct this embedding in a precise way so that it is induced from a general position map from $L \rightarrow \rr^{2n-1}$. 
This implies that a suitable regular neighborhood $M$ of $K$ in $\rr^{2n-1}$ intersects the image of $L$ precisely in a regular neighborhood of $K$ in $L$. 
The boundary of this regular neighborhood of $K$ in $L$ is a disjoint union of links of original vertices of $L$, which by construction are embedded subcomplexes of $\partial M$. 
Thickening these subcomplexes to manifolds inside of $\partial M$ and coning each of them off gives a pseudomanifold with boundary which is homotopy equivalent to $L$.

Here is an example which motivated Theorem \ref{t:main} and Theorem \ref{t:main2}. 

\begin{example}
Let $W_L$ be a right-angled Coxeter group with nerve $L$ a flag simplicial graph. 
The commutator subgroup is finite index in $W_L$ and is the fundamental group of a locally CAT(0) cube complex $P_L$ with all links isomorphic to $L$. 
If $L$ is planar, then we can assume $L$ is a full subcomplex of a flag triangulation of $S^2$, and the Davis complex for the associated right-angled Coxeter group $W_{S^2}$ is a contractible $3$-manifold with a proper, cocompact $W_{S^2}$-action. 
Then $P_L \subset P_{S^2}$ thickens to an aspherical $3$-manifold with boundary. 

If $L$ is not planar, then it is usually the case that $W_L$ is not virtually the fundamental group of an aspherical $3$-manifold \cite{dhw23}.
However, every such $L$ embeds as a full subcomplex into some closed higher genus surface $\Sigma_g$. 
The Davis complex corresponding to a regular neighborhood of $L$ is a contractible $3$-pseudomanifold with isolated singularities and nonempty boundary, and the Davis complex associated to $\Sigma_g$ is a contractible $3$-pseudomanifold with isolated singularities and proper $W_L$-action. 

In higher dimensions one can do a similar construction using the fact that $k$-dimensional simplicial complexes immerse into $\rr^{2k}$, and pulling back a regular neighborhood gives an embedding into a $(2k)$-dimensional manifold. 
\end{example}

\begin{remark}
Bestvina-Kapovich-Kleiner defined the \emph{action dimension} of a finitely generated group $G$ as the minimal dimension of contractible manifold with a proper $G$-action \cite{bkk}.
We propose the \emph{pseudomanifold action dimension} of a finitely generated group as the minimal dimension of contractible pseudomanifold with a proper $G$-action. 
It is a corollary of Stallings' theorem that if $G$ admits a finite $K(G,1)$, then the action dimension of $G$ is $\le 2 \gdim(G)$. 
Similarly, Theorem \ref{t:main} implies the following.
\begin{Cor}
Let $G$ be a group that admits a finite $K(G,1)$. Then the pseudomanifold action dimension of a group $G$ is $\le 2\gdim(G)-1$.
\end{Cor}

Bestvina-Kapovich-Kleiner show that a coarsened version of the van Kampen obstruction gives a lower bound to the action dimension. It would be interesting to find a similar lower bound to the pseudomanifold action dimension. 
Note that Davis-Okun have conjectured that if $P$ is a closed aspherical $n$-pseudomanifold with isolated singularities then the $L^2$-homology of the universal cover $\widetilde P$ vanishes above dimension $\lceil\frac{n}{2}\rceil$ \cite[Conjecture 0.6]{do01}.

Since $2$-dimensional pseudomanifolds are pinched surfaces, there are many examples of $2$-dimensional groups whose pseudomanifold action dimension has to be equal to the bound $2\gdim(G)-1$. In higher dimensions, we do not know how to establish lower bounds. We also do not know whether the notion of the pseudomanifold action dimension does depends on whether we require only isolated singularities.

\begin{question}
    Does there exist a group $G$ whose pseudomanifold action dimension differs depending on whether pseudomanifold is required to have isolated singularities?
\end{question}
\end{remark} 

\subsection*{Related work}

The statement of Theorem \ref{t:main} and its proof was inspired by a talk that Jason Manning gave at BIRS in Fall 2024 on recent work with Ruffoni \cite{mr25}. 
They construct closed, locally CAT$(-1)$ $3$-pseudomanifolds whose fundamental groups are not cubulable; their method is also to thicken a certain $2$-dimensional seed complex (whose fundamental group has Property $(T)$ and is hyperbolic) to a $3$-pseudomanifold with boundary and then apply a version of the reflection group trick. Our construction for thickening $2$-dimensional simplicial complexes is essentially the same as theirs, though with a general seed complex. Therefore, we have not attempted to control any geometry of our complex. On the other hand, our construction allows for more flexible inputs, for instance, it follows from Theorem \ref{t:main2} that there are closed, aspherical $3$-pseudomanifolds whose fundamental groups have unsolvable word problem \cite{CollinsMiller99}. 

Thickenings of $2$-complexes to $3$-pseudomanifolds have also been studied earlier. For instance, Quinn showed that every finite $2$-dimensional complex is homotopy equivalent to a $3$-pseudomanifold with boundary, whose links of vertices are homeomorphic to $\mathbb D^2$, $\mathbb S^2$, or $\rr P^2$ (see e.g.~\cite{HogAngeloniMetzler93}). Since in the Quinn's construction the links can be homeomorphic to $\rr P^2$, the pseudomanifold in not orientable in general. Also, there does not need to be a retraction from the pseudomanifold to the complex.

\subsection*{Acknowledgements}
The first author is supported by NSF grants DMS-1926686 and DMS-2238198, and an NSERC Discovery Grant.
The second author is supported by NSF grants DMS-2203325 and DMS-2505290.

\section{Simplicial Complexes and pseudomanifolds}
\subsection{Simplicial Complexes in $\mathbb R^n$}

We will mainly follow Bryant's notes on PL topology \cite{Bryant2002}.

A \emph{$k$-simplex $\sigma$ in $\rr^n$} (with $n\geq k)$ is the convex hull of $k+1$ affinely independent points $\{v_0, \dots, v_k\}$ in $\rr^k$, i.e.\ where the $k$ vectors $v_1-v_0, \dots v_k-v_0$ are linearly independent. The value $k$ is the \emph{dimension of $\sigma$}.
A \emph{face} of a simplex $\sigma$ is a simplex spanned by any proper subset of the vertices of $\sigma$. 
A \emph{facet} of $\sigma$ is a codimension-one face.

A \emph{simplicial complex} $X$ in $\rr^n$ is a collection of simplices in $\rr^n$ satisfying:
\begin{itemize}
    \item for every simplex $\sigma\in X$ every face $\gt$ of $\sigma$ belongs to $X$.
    \item any non-empty intersection of two simplices $\sigma_1, \sigma_2\in X$ is a face of each of $\sigma_1, \sigma_2$.
\end{itemize}
The union $|X|$ of all simplices of $X$ is a subspace of $\rr^n$.
We abuse the notation and denote by $X$ both this topological space and its combinatorial structure. 
When we write $|X|=|Y|$ we mean that the simplicial complexes $X$ and $Y$ determine the same subspaces of $\rr^n$ but do not necessarily have the same combinatorial structures. 
We say $X$ is a \emph{simplicial complex} if $X$ is a simplicial complex in $\mathbb R^n$ for some $n\in \mathbb N$.
A simplicial complex $X$ viewed combinatorially has a natural decomposition into $X^{(0)}\cup X^{(1)}\cup \dots \cup X^{(i)}\cup \dots$ where $X^{(i)}$ is the set of all simplices of dimension $i$.  

A \emph{simplicial map} $f:X\to Y$ is a map which restricted to each simplex of $X$  is a linear map onto a simplex of $Y$.
Two simplicial complexes $X,Y$ are \emph{isomorphic} if there exists a simplicial map $f: X\to Y$ which is a homeomorphism of subsets of $\rr^n$.

Two simplicial complexes $X,Y$ are \emph{PL-homeomorphic} if there exists subdivisions (see Section~\ref{sec: subdivision} for the definition of a subdivision) of $X$ and $Y$ which are isomorphic. 
A triangulation of a manifold $M$ is a simplicial complex $X$ and a homeomorphism $|X| \rightarrow M$. 
We also say a simplicial complex $X$ is $\emph{PL-
homeomorphic}$ to a manifold $M$ if $X$ is PL-homeomorphic triangulation of $M$. 

A \emph{free face} of a simplicial complex $X$ is a simplex $\gt$ which is a face of exactly one simplex. The \emph{boundary $\partial X$} of $X$ is the union of its free faces, and the \emph{interior $X^\circ$} of $X$ is $X-\partial X\subseteq \mathbb \rr^n$.
In particular, the boundary $\partial \sigma$ of a simplex $\sigma$ is the union of its facets.

Let $v$ be a vertex of $X$. The \emph{star} of $v$ in $X$, denoted $\St_X(v)$, is the subcomplex $\{\gs \in X| v*\gs \in X\}$ where $v*\gs$ is a simplex which is the convex hull of $v\cup \gs$, i.e.\ $v*\gs$ is spanned by $v$ and the vertices of $\gs$.
The \emph{link} of $v$ in $X$, denoted $\Lk_X(v)$, is the subcomplex of $\St_X(v)$ consisting of the simplices of $\St_X(v)$ which do not contain $v$.

\subsection{Abstract simplicial complexes}

At certain points, it is more convenient to work with abstract simplicial complexes. 
Given a vertex set $V$, a simplicial complex $X$ is a collection of subsets of $V$ which is closed under inclusion and contains all singletons. It is naturally a poset under inclusion. 
If the cardinality of $V$ is $n$, then the geometric realization of $X$ is a simplicial complex in $\mathbb{R}^n$ sitting inside the standard simplex $\Delta^{n-1}$, where we include faces of $\Delta^{n-1}$ exactly corresponding to subsets in $X$. 

If $X$ is a simplicial complex and $L$ a subcomplex, then we let $\Cone_X(L)$ denote the simplicial complex obtained from $X$ by adding a new vertex $v$ and the subsets $\tau \cup v$ for $\tau \in L$. If $X$ is a simplicial complex in $\rr^n$, then $\Cone_X(L)$ is naturally a simplicial complex in $\rr^{n+1}$.

We also record the following easy fact:
\begin{lem}\label{lem: retracting cones}
    Let $L,N$ be simplicial complexes in $\rr^n$, and let $H: N\times I\to N$ be a strong deformation retraction of $N$ onto $L$.
    Then there exists a strong deformation retraction of $\Cone_v(N)$ onto $\Cone_v(L)$ whose restriction to $N$ is $H$.
\end{lem}

\subsection{Barycentric subdivision and spine}\label{sec: subdivision}
Let $X, Y$ be simplicial complexes in $\rr^n$. Then $Y$ is a \emph{subdivision} of $X$ if $|X| = |Y|$ and every simplex of $Y$ is contained in a simplex of $X$. 

For every simplex $\gs$ of $X$, let $v_{\gs}$ be a point in the interior of $\gs$ when $\dim \gs >0$ and let $v_{\gs} = \gs$ when $\dim \gs = 0$. 
The corresponding \emph{barycentric subdivision} of $X$ is the simplicial complex $B(X)$ in $\rr^n$ with $$B(X)^{(0)} = \{v_{\gs} : \gs \text{ is a simplex of } X\}$$
and where vertices $v_{\gs_0}, \dots, v_{\gs_i}$ span an $i$-simplex if and only if $\gs_0\subseteq \gs_1\subseteq \dots \subseteq \gs_i$. 
We will also denote $B(X)$ by $X'$.

\begin{remark}
To make such a choice canonical, we can use barycentric coordinates; if $\gs$ is the span of $\{v_0, v_1, \dots v_k\}$, we set $v_\gs$ to $\sum_{i=0}^k\frac{1}{k+1} v_i$. However, in our construction we will require some flexibility in the choice of $v_\gs$
\end{remark}

We will need the following:
\begin{lem}[{\cite[Chap 1, Lem 4]{Zeeman63}}]\label{lem: zeeman}
    Let $K, X$ be simplicial complexes in $\rr^n$ such that $|K|\subseteq|X|$. Then there exists an $r$-fold barycentric subdivision $B^{r}(X)$ of $X$ such that some subdivision $\bar K$ of $K$ is a subcomplex of $B^r(X)$.
\end{lem}

\begin{defn}[Spine $K(X)$]
Given a simplicial complex $X$ and its barycentric subdivision $B(X)$, we define \emph{the spine $K(X)$ of $X-X^{(0)}$} as the subcomplex of $B(X)$ spanned by the vertices $\{v_\gs : \gs \in X^{(k)} \text{ for } k \ge 1\}$, i.e.\ all vertices of $B(X)$ except for those corresponding to $X^{(0)}$.
\end{defn}
For short, we will omit ``of $X-X^{(0)}$'', and refer to $K(X)$ as just the \emph{spine}. 
In the literature $K(X)$ is sometimes referred to as the \emph{dual of the $0$-skeleton of $X$} (see e.g.\ \cite{Bryant2002}). 

\begin{figure}
    \centering
    \includegraphics[width=\linewidth]{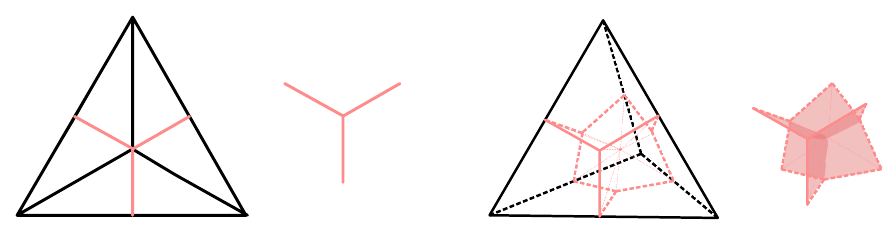}
    \caption{The spine $K(\gs)$ for a $2$- and $3-$simplex, viewed as a subcomplex of the barycentric subdivision $B(\gs)$ or as its own simplicial complex.}
    \label{fig:K(X)}
\end{figure}

\begin{lem}\label{lem: spine}
The spine $K(X)$ is a deformation retract of $X - X^{(0)}$.
\end{lem}

\begin{proof}
First note that $K(X) = \bigcup_{v\in X^{(0)}} \Lk_{B(X)}(v)$ 
 and $B(X) = \bigcup_{v\in X^{(0)}} \St_{B(X)}(v)$. 
 The different stars possibly intersect only in their links. 
For each vertex $v \in X^{(0)}$,  there is a deformation retraction $$r_v: \St_{B(X)}(v) - \{v\} \rightarrow \Lk_{B(X)}(v).$$ Thus we can define $r:[0,1]\times (X - X^{(0)}) \to K(X)$
given piecewise as $r(p) = r_v(p)$ if $p\in \St_{B(X)}(v)$.
\end{proof}

\subsection{Relative barycentric subdivision}
Let $K\subseteq X$ be a subcomplex.
For each simplex $\gs$ in $X$ not contained in $K$, let $v_\gs$ be a point in the interior of $\gs$ when $\dim\gs>0$, and let  $v_\gs=\gs$ when $\dim\gs = 0$.
The corresponding \emph{barycentric subdivision of $X$ relative to $K$} is the simplicial complex $B(X,K)$ in $\mathbb R^n$ with 
$$ B(X,K)^{(0)} = \{v_\gs:\gs \text{ is a simplex in }X \text{ not contained in }K\} \cup K^{(0)}
$$
and where vertices $v_{\gs_1}, \dots, v_{\gs_i}, w_1, \dots, w_j$ span a simplex if $w_1, \dots, w_j$ span a simplex $\gt$ in $K$, and $\gt\subseteq v_{\gs_1}\subseteq \dots \subseteq v_{\gs_i}$. See Figure~\ref{fig:U(K(X))}.

\subsection{Regular neighborhoods}
Let $X$ be a simplicial complex in $\mathbb R^n$. A subcomplex $K\subseteq X$ is \emph{full} if a simplex $\gs$ of $X$ belongs to $K$ if and only if all of its vertices belong to $K$. 
Let $K$ be a full subcomplex of a simplicial complex $X$. 
The \emph{simplicial neighborhood} $N(K,X)$ of $K$ in $X$ is the subcomplex
$$\bigcup_{v \in K} \St_X(v).$$

In other words, $N(K,X)$ is the smallest subcomplex of $X$ containing every simplex which intersects $L$. The \emph{boundary} of $N(K,X)$, denoted $\dot N(K,X)$, is the subcomplex of $N(K,X)$ containing simplices which do not intersect $K$.

A \emph{regular neighborhood of $K$ in $X$} is the simplicial neighborhood $N(K, X')$ for $X'$ a barycentric subdivision of $X$ relative to $K$. See Figure~\ref{fig:U(K(X))}. 

\begin{lem}\label{lem: boundary of a regular neighborhood of the spine}
Let $K=K(X)$ be the spine of a simplicial complex $X$. Then the boundary $\dot N$
of the regular neighborhood $N$ of $K$ in $X$ is the disjoint union of $\Lk_{X''}(v)$ over all $v\in X^{(0)}$.
\end{lem}

\begin{proof}

A vertex in $X''$ corresponds to a chain of simplices in $X$.  
The vertices in $\Lk_{X''}(v)$ for $v \in X^{(0)}$ are chains with minimal term a simplex $\tau$ properly containing $v$. 

Similarly the vertices in $B(X'', K)^{(0)} - K^{(0)}$ correspond to simplices in $X'$ not contained in $K$, i.e. chains of simplices with minimal element a simplex not in $K$. One of these spans an edge with a vertex $k \in K^{(0)}$ if and only if $k$ is contained in the minimal element, i.e. the minimal element is a simplex $\tau$ properly containing some vertex $v \in X^{(0)}$.
Therefore, the vertices of $\bigcup \Lk_{X''}(v)$ and $\dot N(K, X'')$ are the same. 

We claim that $\dot N(K, X'')$ is a full subcomplex, since $\bigcup \Lk_{X''}(v)$ is also full we will be done. A simplex $\gs$ in $X''$ correponds to nested chains of simplices of $X$. Since the chains are nested, each minimal element of a chain in $\gs^{(0)}$ contains a simplex $\gs_0$ of $X$. If the vertices of $\gs$ are in $\dot N(K, X'')$, then by the above $\gs_0$ is not a vertex of $v$. Therefore, $\gs$ and $(\gs_0 - \{v\})$ spans a simplex in $N(K, X')$, and hence $\gs \subset \dot N(K, X')$
\end{proof}

Any two regular neighborhoods corresponding to barycentric subdivisions $X_1'$ and $X_2'$ relative to $K$ are canonically homeomorphic  via a homeomorphism that fixes $K$.
Regular neighborhoods are also preserved under taking subdivisions:
\begin{lem}[{\cite[Prop 3.1]{Bryant2002}}]\label{l:neighborhoodsubdivision}
Suppose that $K$ is a full subcomplex of a simplicial complex $X$ in $\rr^n$, and suppose $X_1$ is a subdivision of $X$ inducing a subdivision $K_1$ of $K$. 
Then there is a regular neighborhood $N$ of $K$ in $X$ and a regular neighborhood $N_1$ of $K_1$ in $X_1$ so that 
$|N| = |N_1|$. 
\end{lem}

A \emph{combinatorial $d$-disc} is a simplicial complex PL-homeomorphic to a $d$-simplex, and a  \emph{combinatorial $d$-sphere} is a simplicial complex PL-homeomorphic to the boundary of a $(d+1)$-simplex.
A simplicial complex $X$ is a \emph{combinatorial $d$-manifold} if the link of each $p$-simplex is either a combinatorial $(d-p-1)$-sphere, or a combinatorial $(d-p-1)$-disc for $p=0, \dots, d$. The simplices with the links of the second kind form a subcomplex called the \emph{boundary} $\partial X$ of $X$. This subcomplex is itself a combinatorial $(d-1)$-manifold without boundary. If $X$ is a combinatorial $d$-manifold, then $|X|$ is obviously a topological $d$-manifold with (possibly empty) boundary $\partial |X| = |\partial X|$.

\begin{thm}[{\cite[Thm 3.6]{Bryant2002}}]\label{thm: boundary of a regular neighborhood}
    Suppose $K$ is a subcomplex of a combinatorial manifold $M$. Then a regular neighborhood $N$ of $K$ in $M$ is a combinatorial manifold. If $K$ is in the interior of $M$ then $\partial |N| = |\dot N|$.
\end{thm}

\begin{defn}
    Let $K$ be a full subcomplex of $X$. 
    Given $\epsilon > 0$, the \emph{regular $\epsilon$-neighborhood} $N_\epsilon$ of $K$ in $X$ is a simplicial neighborhood constructed as follows. Since $K$ is full in $X$, the simplicial map $f: X \rightarrow [0,1]$ which sends the vertices of $K$ to $0$ and the other vertices to $1$ has the property that $f^{-1}(0) = K$.
    For any simplex $\gs$ of $N(K,X)$ which is not a simplex of $K$, we choose $v_\gs \in  \gs^\circ \cap f^{-1}(\epsilon)$. For any other simplex $\gs$ of $X$ which is not in $K$ choose $v_\gs\in \gs^\circ$ arbitrarily. This yields a barycentric subdvision $X'$ of $X$ mod $K$. Let $N_\epsilon = N(K, X')$.
\end{defn}

\begin{figure}
    \centering
    \includegraphics[width=0.7\linewidth]{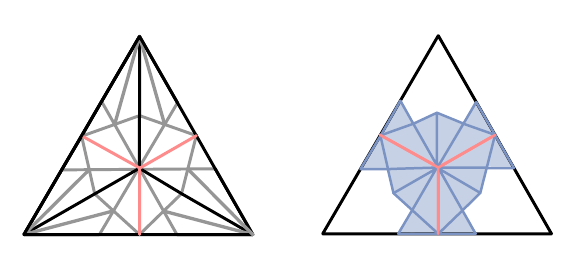}
    \caption{A barycentric subdivision of a $2$-simplex $\gs$ relative the spine $K(\gs)$, and the corresponding regular neighborhood $N(K(\gs))$.}
    \label{fig:U(K(X))}
\end{figure}

\subsection{Pseudomanifolds}
There are various definitions of pseudomanifolds (with and without boundary) in the literature. The weakest definition is the following.

\begin{defn}
    A finite connected $d$-dimensional simplicial complex $P$ is a \emph{$d$-dimensional pseudomanifold with boundary} if 
    \begin{enumerate}
        \item $P$ is \emph{pure}, i.e.\ $P$ is a union of its $d$-simplices, and
        \item each $(d-1)$-simplex is a face of one or two $d$-dimensional simplices.
    \end{enumerate}
    In this case, the \emph{boundary} $\partial P$ is the union of the $(d-1)$-simplices which are faces of single $d$-dimensional simplices. When $\partial P=\emptyset$, then we say $P$ is a \emph{closed pseudomanifold}.
\end{defn}

It follows that $\Lk_P(\gs)$ is a disjoint union of $(d-k-1)$-dimensional pseudomanifolds for every $k$-simplex $\gs$.
The definition often includes the assumption that a pseudomanifold is \emph{gallery connected} (\cite{Davis08}, also referred to as \emph{strongly connected} in other sources, e.g.\ \cite{Bjorner}): 
for every pair of $d$-simplices $\gs$ and $\gs'$ in $P$, there is a sequence of $d$-simplices $\gs = \gs_0, \gs_1,\dots \gs_k = \gs'$ such that the intersection $\gs_i\cap \gs_{i+1}$ is an $(d-1)$-simplex for all $i = 0, ..., k-1$. One downside is a link in a gallery connected pseudomanifold does not need to be gallery connected. 
The pseudomanifold that we construct in Theorem~\ref{t:main} is gallery connected, so our main results remains unchanged if we include this assumption. 

\begin{defn}
    A $d$-dimensional pseudomanifold $P$ with boundary has \emph{isolated singularities} if for every $k$-simplex $\gs$ with $k>0$, $\Lk_P(\gs)$ is PL-homeomorphic to the boundary of a $(d-k)$-simplex when $\gs\not\subseteq \partial P$, and $\Lk_P(\gs)$ is PL-homeomorphic to a $(d-k-1)$-simplex when $\gs\subseteq \partial P$.
\end{defn}
It follows that in a pseudomanifold with isolated singularities $P$, each vertex link $\Lk_P(v)$ is a combinatorial manifold:

\begin{lem}
    A pseudomanifold has isolated singularities if and only if the vertex links are combinatorial manifolds. 
\end{lem}

\begin{proof}

Given a vertex $v \in P$ and $\gs \in \Lk_P(v)$, we have that $\Lk_P(\gs\ast v)$ is isomorphic to $\Lk_{\Lk_P(v)}(\gs)$. The conclusion immediately follows. 
\end{proof}

An \emph{orientation} of a $k$-simplex $\gs = \{v_0, \dots, v_k\}$ is an equivalence class of orderings of the vertices of $\gs$, where two orderings are equivalent if they differ by an even permutation. We write $(v_0, \dots, v_k)$ to represent the orientation given by the listed ordering, and denote the opposite orientation by $-(v_0, \dots, v_k)$. 

An orientation $(v_0, \dots, v_k)$ of $k$-simplex induces the orientation of its facets: the facet excluding the vertex $v_i$ has \emph{induced orientation} $(-1)^i(v_0, \dots, v_{i-1}, v_{i+1}, \dots v_k)$

\begin{defn}
    A $d$-dimensional pseudomanifold (with or without a boundary) $P$ is \emph{orientable} if there is an orientation of each $d$-simplex such that the induced orientations on each $(d-1)$-simplex contained in two $d$-simplices are opposite. Equivalently, $H_d(P, \partial P; \zz) \ne 0$. 
\end{defn}

Note that the links in an orientable pseudomanifold are orientable.

\section{Embedding simplicial complexes into pseudomanifolds}

In this section, we prove Theorem \ref{t:main}.

\subsection{Embedding a simplicial complex in Euclidean space}
A countable set $S$ of points in $\rr^n$ is said to be \emph{in general position},  if every subset $\{v_0 , v_1, \dots, v_k\}$ of $S$ with $k \le n$ spans a $k$-simplex.
More generally, for a simplicial complex $X$ we say that a map $f: X\to \rr^n$ which is linear on simplices is \emph{a general position map}, if $f$ restricted to the $0$-skeleton $X^{(0)}$ is an embedding and the set $f(X^{(0)})\subseteq \rr^n$ is in general position.

Suppose $f: X \rightarrow \rr^n$ is a continuous piecewise linear function. 
The \emph{singular set} of $f$ is the subset $S(f) = \text{Cl}\{x \in X : f^{-1} (f(x)) \ne \{x\}\}$. If $X$ is compact, then $S(f)$ intersects each simplex of $X$ in a finite union of linear subspaces. We define $\dim S(f)$ to be the largest dimension of such a  subspace. 

\begin{lem}\label{lem: dimension of the singular set}
    Let $X$ be a simplicial complex, and let $f:X\to \rr^n$ be a general position map. Suppose that $\gs_1, \gs_2$ are two simplices of $X$ of dimension $d_1$ and $d_2$, and let $d_3 = \dim\gs_1\cap \gs_2$ ($d_3=-1$ when $\gs_1\cap \gs_2 =\emptyset$). Then 
    $$\dim S(f_{|\gs_1\cup\gs_2}) \leq d_1+d_2-n \qquad \text{ if }\quad d_1 + d_2 - d_3 \ge n$$ and $S(f_{|\gs_1\cup\gs_2}) = \emptyset$ otherwise. 
\end{lem}

\begin{proof}
Let $\gs_1$ be the simplex spanned by vertices $v_0, \dots, v_{d_1}$, and $\gs_2$ be the simplex spanned by vertices $w_0, \dots, w_{d_2}$, where $v_i = w_i$ for $i=0, \dots, d_3$. 

If $|\gs_1\cup \gs_2| = d_1+d_2-d_3 +1 \le n+1$, then the vertices of $\gs_1 \cup \gs_2$ are mapped by $f$ to a set spanning an $n$-dimensional simplex in $\rr^n$, so $f_{|\gs_1\cup\gs_2}$ is an embedding. In particular,  $S(f_{|\gs_1\cup\gs_2}) = \emptyset$. 

If $|\gs_1\cup \gs_2| = d_1+d_2-d_3 +1$ is greater or equal to $n+1$, then  $\dim (\spn f(\gs_1\cup \gs_2))=n$, and so
\begin{align*}
    &\dim (S(f_{|\gs_1\cup\gs_2})) = \dim (Cl(f(\gs_1^\circ)\cap f(\gs_2^\circ))) \leq \dim (\spn\{f(\gs_1^\circ)\cap f(\gs_2^\circ)\}) =\\
    &= \dim f(\gs_1) + \dim f(\gs_2) - \dim (\spn f(\gs_1\cup \gs_2))= d_1+ d_2 -n.
\end{align*}
\end{proof}

We have the following consequences of Lemma~\ref{lem: dimension of the singular set}.
\begin{Cor}\label{cor: general position embedding}
    Let $X$ be a finite $d$-dimensional simplicial complex, and let $f:X\to \rr^{2d+1}$ be a general position map. Then $f$ is an embedding. 
\end{Cor}

\begin{Cor}\label{cor: general position general dimension}
    Suppose $X$ is a finite $d$-dimensional simplicial complex and $f: X \rightarrow \rr^n$ is a general position map. Then $\dim(S(f)) \le 2d-n$.
\end{Cor}

\subsection{Embedding the spine in Euclidean space}
Given two affine subspaces $V, W\subseteq \rr^k$, $V$ and $W$ \emph{intersect transversely} if $V,W$ together span $\rr^k$. In particular, when $\dim V + \dim W = k$, then their intersection is a single point, and $\rr^k$ can be expressed as the direct sum $V\oplus W$.

Similarly, given convex subsets $A, B, C\subseteq \rr^k$, we say that the intersection of $A$ and $B$ is transversal in $C$, if $A\cap B$ is nonempty and $\spn(A)$ and $\spn(B)$ intersect transversely in $\spn(C)$.

\begin{lem}\label{lem: generic choice of v}
    Let $\gs \subseteq \rr^k$ be a $k$-simplex with a choice of $K(\partial \gs)$ inside of $\partial \gs$.
    Let $\mathcal P = \bigcup\{p_1, \dots, p_s\}$ be a finite collection of points in the interior of $\gs$, and $\mathcal L = \bigcup \{\ell_1, \dots, \ell_t\}$ be a finite collection of line segments in $\gs$ not contained in $\partial \gs$. 
    For $v \in \gs^\circ$, let $K_v = \Cone_v(K(\partial \gs))\subseteq \mathbb R^k$. Then the set 
    $$\{v\in \gs^\circ \mid K_v\cap \mathcal P \neq \emptyset \text { or }  K_v \text{ intersects } \mathcal L \text { not transversely in }\gs \} $$ 
    is contained in a union of finitely many codimension one subspaces of $\rr^k$.
\end{lem}

\begin{proof}
The set $$\{v\in \rr^k \mid K_v\cap \mathcal P \neq \emptyset \text { or }  K_v \text{ intersects } \mathcal L \text { non-transversely in } \gs\} $$ 
is a union of finitely many sets of the form
$$X_{p,\gt} :=\{v\in \rr^k \mid p\in \Cone_v(\gt)\} $$ 
for $p\in \mathcal P$, and a face $\gt$ of $K(\partial \gs)$
and finitely many sets of the form
$$Y_{\ell, \gt} := \{v\in \rr^k \mid \ell \text{ intersects } \Cone_v(\gt) \text{ non-transversely in } \gs\} $$ 
for $\ell\in \mathcal L$, and a face $\gt$ of $K(\partial \gs)$.
We show that each of those sets is contained in a codimension one subspace of $\rr^k$.

Fix $p\in \mathcal P$ and a face $\gt$ of $K(\partial \gs)$. The set of $v\in \rr^k$ such that $p\in \Cone_v(\gt) $ is contained in the set 
$\{v\in \rr^k \mid p\in \spn(\{v\}\cup h(\gt))\}$. But the condition that $p\in \spn(\{v\}\cup h(\gt))$ is equivalent to the condition that $v\in \spn(\{p\}\cup \gt)$. Since by assumption $p\notin \spn{\gt}$ and $\dim\gt \leq k-2$, we have $\dim (\spn(\{p\}\cup \gt)\}) = \dim \gt + 1 \leq k-1$. 

Similarly, if $\Cone_v(\gt)$ intersects a line $l$ non-transversely, then $v$ is in the codimension-one hyperplane spanned by $\gt$ and $l$. 
\end{proof}

\begin{remark}
    In our application of Lemma~\ref{lem: generic choice of v} in the proof of Proposition~\ref{prop: existence of h} $\mathcal P\subseteq \mathcal L$.
\end{remark}

\begin{prop}\label{prop: existence of h}
Let $X$ be a finite $d$-dimensional complex and $f: X \rightarrow \mathbb{R}^{2d-1}$ be a general position map. Then there exists a choice of spine $K(X) \subset X$ such that $f$ restricted to $K(X)$ is an embedding. 

Moreover, there exists $\epsilon>0$ such that the embedding extends to the regular $\epsilon$-neighborhood $N_{\epsilon}$ of $K(X)$ in $X$.
\end{prop}

\begin{proof} 

For $i=1, \dots, k$, let $\gs_1, \dots, \gs_k$ be the $d$-simplices of $X$. 
In this proof, there are various objects that we denote by a pair of subscripts $ij$ where $i,j\in \{1, \dots, k\}$. We follow a convention that any such object whose first subscript is $i$ is a subset of a simplex $\gs_i$.\\

We first show how to choose $K(X)$ such that $f$ restricted to $K(X)$ is an embedding. 
The spine $K(X)$ is the union 
$$K(X) = K(X^{(d-1)})\cup \bigcup_{i=1}^k K(\gs_i),$$ 
and each $K(\gs_i)$ is naturally isomorphic to $\Cone_{v_{\gs_i}}(K(\partial \gs_i))$.
By Corollary~\ref{cor: general position embedding} the restriction $f_{|X^{(d-1)}}$ to the $(d-1)$-skeleton of $X$ is an embedding. 
By Corollary~\ref{cor: general position general dimension} $\dim S(f)\leq 1$. In particular, for any $\gs_i, \gs_j$ with $i\neq j$ such that $f(\gs_i^\circ)\cap f(\gs_j^\circ)\neq \emptyset$, the intersection $f(\gs_i^\circ)\cap f(\gs_j^\circ)$ consists of either a single vertex or a line segment  intersecting the faces of $\gs_1$ and $\gs_2$ transversely. Let $\ell_{ij}$ be its preimage in $\gs_i$ under $f_{|\gs_i}$. See Figure~\ref{fig:intersecting sigmas}. If $f(\gs_i^\circ)\cap f(\gs_j^\circ)= \emptyset$, let $\ell_{ij} = \emptyset$. 
Let $\mathcal L_i = \bigcup_{j\neq i} \ell_{ij}\subseteq \gs_i$, and $\mathcal L = \bigcup \mathcal L_i$. Note that $\mathcal L$ is a finite union.

\begin{figure}
    \centering
    \includegraphics[width=\linewidth]{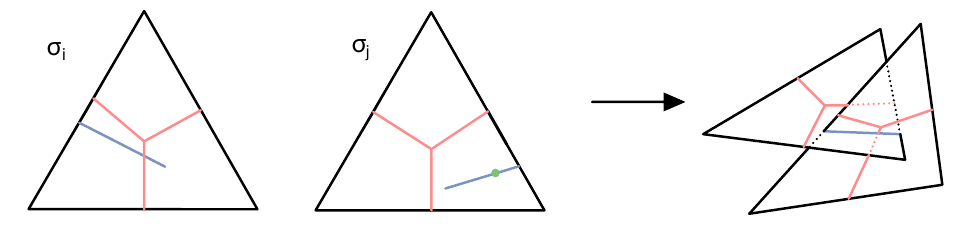}
    \caption{Two simplices $\sigma_i, \sigma_j$ map to $\rr^{2d-1}$ with an intersection along a line segment.
    The blue line segments are $\ell_{ij}\subseteq \sigma_i$ and  $\ell_{ji}\subseteq \sigma_j$. The distinguished green point in $\ell_{ji}$ forms $\mathcal P_{ji}$. The spine $K(\gs_j)$ is always chosen so that $K(\sigma_j)\cap \mathcal P_{ji}=\emptyset$. 
    }
    \label{fig:intersecting sigmas}
\end{figure}

We now choose a  barycenter $v_\gs$ for each simplex $\gs$ of $X$ with $\dim \gs > 0$.
If $\dim \gs<d$, we can choose $v_{\gs}$ to be an arbitrary point in the interior of $\gs$ such that $f(v_{\gs})$ is disjoint from $\mathcal L$. 
We pick $v_{\gs_i}$ for $i=1 , \dots, k$ inductively, so that $K(\gs_i)$ intersects the line segments $\ell_{ij}$ transversely. 
Suppose we have already made choices of $v_{\gs_1}, \dots, v_{\gs_{i-1}}$. 
For $1\leq j<i\leq k$, let $\mathcal Q_{ij}\subseteq \ell_{ij}\subseteq \gs_i$ be the finite set of points that belong to the intersection $\ell_{ij}\cap K(\gs_i)$. Indeed, $Q_{ij}$ is a finite set of points, since $K(\gs_i)$ is a codimension one subspace of $\gs_i$, and each $\ell_{ij}$ is a line segment intersecting $K(\gs_i)$ transversely.
Let $\mathcal P_{ij}\subseteq \ell_{ij}\subseteq \gs_i$ be the finite set of points such that $f(\mathcal P_{ij}) = f(\mathcal Q_{ji})$, i.e.\ for each point $p\in \mathcal P_{ij}$ its image $f(p)$ belongs to the intersection $f(\gs_i)\cap f(K(\gs_j))$. In other words, $\mathcal P_{ij} = f_{|\gs_i}^{-1}(f(\mathcal Q_{ji}))$.  See Figure~\ref{fig:intersecting sigmas}.

Let $\mathcal P_i= \bigcup_{j<i} \mathcal P_{ij}$. 
By Lemma~\ref{lem: generic choice of v} there exists a choice of $v_{\gs_i}$ such that $K(\gs_i)\cap \mathcal P_i = \emptyset$ and 
$K(\gs_i)$ intersects each line of $\mathcal L_i$ transversely.
Thus we have chosen $v_{\gs_i}$ for $i=1, \dots, k$, such that $f$ restricted to $K(X)$ is an embedding.\\

It remains to show that $f$ restricted to the regular $\epsilon$-neighborhood of $K(X)$ is still an embedding, for a sufficiently small $\epsilon>0$. 

Choose $\delta > 0$ so that $\delta< d_{\mathbb{R}^n}(f(K(\gs_i)), f(K(\gs_j)))$ for any $d$-simplices $\gs_i,\gs_j$ which intersect in at most one vertex. 
Since $f$ is uniformly continuous, there is $\epsilon > 0$ so that for any $x,y\in X$ $d_X(x,y) < \epsilon$ implies $d_{\mathbb{R}^n}(f(x), f(y)) < \delta/2$.

Let $x_1 \ne x_2 \in N_\epsilon$ such that for $i=1,2$, $x_i\in \gs_i$ for some $d$-simplex $\gs_i$.
If $\gs_1, \gs_2$ have a common face of dimension $>0$, then by Lemma~\ref{lem: dimension of the singular set}, $f$ is an embedding on $\gs_1\cup \gs_2$ as $d+d-d_3\geq 2d-1$. In particular, $f(x_1)\neq f(x_2)$.

Now suppose $\gs_1, \gs_2$ have at most one common vertex. Since the distance from $f(x_i)$ to $f(K(\gs_i))$ is at most $\delta/2$ we have
\begin{align*}
    &d_{\mathbb{R}^n}(f(x_1), f(x_2)) \geq \\
    &d_{\mathbb{R}^n}(f(K(\gs_i)), f(K(\gs_j))) -  d_{\mathbb{R}^n}(f(x_1), f(K(\gs_1))) - d_{\mathbb{R}^n}(f(x_2), f(K(\gs_2))) >\\
    &  \delta - \delta/2 - \delta/2 = 0.
\end{align*}
In particular, $f(x_1)\neq f(x_2)$. 
\end{proof}

\subsection{Embedding a simplicial complex in a pseudomanifold with boundary}

\main*

\begin{proof}

\item
(1) \emph{ Construction of $P(X)$:}
Let $X$ be a $d$-dimensional simplicial complex. Let $f: X \rightarrow \rr^{2d-1}$ be any general position map.
By Proposition \ref{prop: existence of h}, $f$ restricts to an embedding of $N_{\epsilon}$ of $K(X)$ in $X$ for some $\epsilon>0$ and appropriate choice of $K(X)$ inside of $X$.  
Thus we can view $N_{\epsilon}$ as a simplicial complex in $\rr^{2d-1}$ (to simplify the notation we do not write $f(N_{\epsilon})$). This step is illustarated by the first picture of Figure~\ref{fig:steps}. By Lemma~\ref{lem: zeeman} some subdivision $\hat N$ of $N_{\epsilon}$ is a full subcomplex of some triangulation $R$ of $\rr^{2d-1}$ (i.e.\ $|R|= \rr^{2d-1}$). We denote the induced subdivision of $K(X)$ by $\hat K$. In particular, $\hat K$ is a subcomplex of $\hat N$.
\begin{figure}
    \centering
    \includegraphics[width=\linewidth]{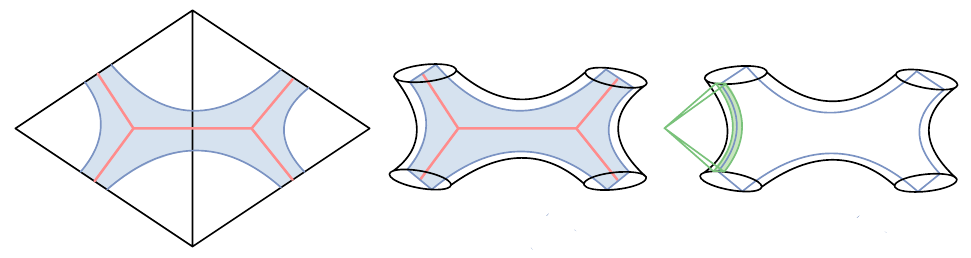}
    \caption{Construction of $P(X)$: (1) choice of a regular neighborhood of the spine in the complex $X$, (2) a regular neighborhood $M$ of the spine in $\rr^n$, (3) coning-off of submanifolds of the boundary of $M$.}
    \label{fig:steps}
\end{figure}

Let $\epsilon'>0$ be sufficiently small so that the regular $\epsilon'$-neighborhood $M$ of $\hat K$ in $R$ satisfies $|M\cap \hat N|\subseteq |\hat N^\circ|$. By Theorem~\ref{thm: boundary of a regular neighborhood} $M$ is a combinatorial $(2d-1)$-manifold with boundary, and $|\dot M| = \partial |M|$. See the second picture of Figure~\ref{fig:steps}.

Since $\hat N$ is a subcomplex of $R$, the regular $\epsilon'$-neighborhood $L$ of $\hat K$ in $\hat N$ is naturally a subcomplex of $M$, and $\dot L$ is an embedded subcomplex of $\partial M$.
By Lemma~\ref{l:neighborhoodsubdivision} and Lemma~\ref{lem: boundary of a regular neighborhood of the spine}, $\dot L$ is a disjoint union $\bigsqcup_{v\in X^{(0)}} L_v$ where each $L_v$ is PL-homeomorphic to $\Lk_X(v)$. 
A regular neighborhood $N_v$ of $L_v$ in $\partial M$ is a codimension zero combinatorial submanifold of $\partial M$. See the last picture of Figure~\ref{fig:steps}.

We define $P(X)$ as $M\cup \bigsqcup_{v\in X^{(0)}} \Cone(w_v, N_v)$, i.e.\ $M$ with each $N_v$ coned off. Clearly $P(X)$ is a simplicial complex (in $\rr^n$ for some $n$ depending on the number on cone points). 
We refer to the new vertices $w_v$ as \emph{cone vertices} of $P(X)$. \\

\noindent (2) \emph{$P(X)$ is a $(2d-1)$-pseudomanifold with isolated singularities:}

Let us analyze the links of vertices in $P(X)$. 
If $u=w_v$ is a cone vertex, then $\Lk_{P(X)}(v)$ is the regular neighborhood $N_v$, which is a combinatorial $(2d-2)$-manifold with boundary. 
If $u$ is not in the link of a cone vertex, then $u\in M$ and $\Lk_{P(X)}(u)$ is equal to $\Lk_M(u)$. 
Finally, if $u$ is in the link of a cone vertex $w_v$, then $u \in \partial M$ and $\Lk_{P(X)}(u)$ is obtained from $\Lk_M(u)$ by coning off $\Lk_{N_v}(u)$. The latter is either PL-homeomorphic to the boundary of a $(2d-2)$-simplex (if $u \in N_v - \partial N_v)$, or PL-homeomorphic to a $(2d-3)$-simplex (if $u \in \partial N_v$). In the first case, $\Lk_{P(X)}(u)$ is a combinatorial $(2d-2)$-sphere, and in the second, it is a $(2d-2)$-disc. This proves that $P(X)$ is a $(2d-1)$-pseudomanifold with isolated singularities.\\

\noindent (3) \emph{$P(X)$ is orientable:}
Since $M$ is a codimension-zero manifold with boundary in $\rr^{2d-1}$, each of its simplices can be oriented consistently with some fixed orientation of $\rr^{2d-1}$. Then for any $(2d-2)$-dimensional simplex $\tau$ of $M$ which is not in $\partial M$, the two orientations on $\tau$ induced by the orientations of the two $(2d-1)$-simplices containing $\tau$ as a face must be reversed. Each $N_v$ is a $(2d-2)$-dimensional manifold with boundary, and the orientation of simplices of $M$ induces an orientation on each $(2d-2)$-simplex $\tau$ of $N_v$. Let the orientation of $\tau$ be given by ordering $(v_0, \dots, v_{2d-2})$. We orient each simplex $v*\tau$ of $\Cone_{w_v}(N_v)$ as $-(w_v, v_0, \dots, v_{2d-2})$. Then the orientation induced on $\tau$ by the orientation of $v*\tau$ is $-(v_0, \dots, v_{2d-2})$, as required. It remains to check that for any $(2d-2)$-simplex containing $w_v$ and which is contained in two $(2d-1)$-simplices, the two orientations induced by the $(2d-1)$-simplices are opposite.
Let $\delta$ be a $(2d-3)$-simplex in $N_v-\partial N_v$ spanned by vertices $v_1, \dots, v_{2d-2}$. Since $N_v$ is an orientable combinatorial manifold with boundary, $\delta$ belongs to two $(2d-2)$-simplices $\tau, \tau'$, and their orientations can be represented by the orderings of their vertices $(v_0, \dots, v_{2d-2})$ and $-(v_0', \dots, v_{2d-2})$ as they must induce the opposite orientations on $\delta$. Thus $w_v*\tau$ and $w_v*\tau'$ also induce opposite orientations on $w_v*\delta$.\\

\noindent (4) \emph{$X$ is a strong deformation retract of $P(X)$:}
Any regular neighborhood $M$ admits a strong deformation retraction onto $K(X)$, which by Lemma~\ref{lem: retracting cones} can be extended to the strong deformation retraction of $P(X)$ onto $X$.
\end{proof}

\begin{remark}
    The assumption that $d\geq 2$ is clearly necessary, as $1$-dimensional pseudomanifolds are $1$-dimensional manifolds. If we try to follow the steps of our construction of $P(X)$ in the case of $d=1$, we get a finite discrete set of intervals as $M$, and disconnected subsets $N_v$. In that case, coning off sets $N_v$ does not yield a pseudomanifold (but in fact the original complex $X$). 
\end{remark}

\section{
Reflection group trick for pseudomanifolds}

We now use a version of Davis's reflection group trick to embed a pseudomanifold with isolated singularities $D$ into a closed pseudomanifold $P$ with isolated singularities. The ``usual" reflection group trick turns an aspherical manifold $N$ with boundary into a closed aspherical manifold $M$ which retracts onto $N$. The general version of the trick we use is described in \cite[Section 11.5]{Davis08}. 
\begin{remark}
    Davis also considers a version of the reflection group trick for pseudomanifolds with boundary in \cite[Section 13.4]{Davis08}. However, Davis additionally requires that every point in $\partial P$ has a neighborhood homeomorphic to $U\times[o,\epsilon)$ where $U$ is some open neighborhood of the point in $\partial P$. This condition fails in general for the cone points in $P(X)$ as constructed in the proof of Theorem~\ref{t:main}. Thus we cannot directly quote this version. 
\end{remark}

\begin{defn}
    A \emph{mirror structure} on a CW-complex $Y$ consists of an index set $S$ and a family of subcomplexes $\{Y_s\}_{s \in S}$.
    The subcomplexes $Y_s$ are the \emph{mirrors} of $Y$.
    We say $Y$ is a mirrored space over $S$.
\end{defn}

Let $Y$ be a mirrored space over $S$, and let $W$ be a right-angled Coxeter group with vertex set $S$.
Given $y\in Y$, let $S(y) := \{s \in S\mid y \in Y_s\}$, and $W_{S(y)}$ be the Coxeter subgroup generated by $S(y)$.
Given any Coxeter group generated by $S$, let $\cu(W, Y )$ be the \emph{basic construction} associated to this data:
\[
\cu(W, Y) := (W \times Y) /\sim,
\]
where $(w,y) \sim (w',y')$ if and only if $y = y'$ and $wW_{S(y)} = w'W_{S(y)}$.

Let $Y$ be a simplicial complex, and $L$ a flag subcomplex.
This produces a mirror structure on the barycentric subdivision $Y'$, where $S = L^{(0)}$. For $s\in S$, the mirror $Y_s$ is the star of the vertex $s$ in the barycentric subdivision $L'$ of $L$, i.e. $Y_s$ is the geometric realization of the poset of simplices of $L$ containing $s$. 

Let $W_L$ be the right-angled Coxeter group with respect to $L$. 
We can then form the basic construction $\mathcal{U}(W_L,Y')$. Topologically, there is no difference between $\mathcal{U}(W_L,Y')$ and $\mathcal{U}(W_L,Y)$, but later we will want a succinct description of the links of vertices, and $Y'$ helps. 
If $\Gamma$ is a torsion-free finite index subgroup of $W_L$, then the output of the reflection group trick is the space $\mathcal{U}(W_L,Y')/\Gamma$. 

The universal cover of $\mathcal{U}(W_L,Y')/\Gamma$ is itself a basic construction $\mathcal{U}(W_{\widetilde L},\widetilde Y')$, where $\widetilde Y'$ is the universal cover of $Y'$ and $W_{\widetilde L}$ is the right-angled Coxeter group based on the lifted flag triangulation $\widetilde L$ \cite[pg. 168]{Davis08}.
In particular, if $Y$ is aspherical then $\mathcal{U}(W_{\widetilde L},\widetilde Y')$ is contractible, so $\mathcal{U}( W_L, Y')/\Gamma$ is aspherical. The orbit map $\mathcal{U}( W_L, Y') \rightarrow Y'$ induces a retraction from $\mathcal{U}( W_L, Y')/\Gamma \rightarrow Y'$.

\begin{prop}\label{prop:reflection group trick}
Suppose that $P$ is a pseudomanifold with isolated singularities obtained from a compact combinatorial $d$-manifold $M$ with boundary by attaching cones  to some number of disjoint codimension zero PL-submanifolds with boundary $(N_i, \partial N_i)$ of $\partial M$. Suppose that the induced triangulation $L$ on $\partial P$ is flag. 

Then for any torsion-free finite index subgroup $\Gamma\subseteq W_L$, the space $Q = \mathcal{U}(W_L,P')/\Gamma$ for any is a closed pseudomanifold with isolated singularities. Furthermore, if $P$ is aspherical, then $Q$ is aspherical. 
\end{prop}

\begin{proof}
The triangulation $L$ induces a mirror structure on $P'$ with mirrors $\St_{L'}(s)$ for $s\in S = L^{(0)}$, as above.
We prove that $\mathcal{U}(W_L,P')$ is a pseudomanifold with isolated singularities, which implies the result as this is preserved by the $\gG$-action. 
Since $W_L$ acts on $\mathcal{U}(W_L,P')$ with strict fundamental domain $P'$, it suffices to consider vertices in the chamber $P'$. 
If $v \notin L = \partial P'$, then $P'$ is the unique chamber containing $v$, and $\Lk_{P'}(v)$ is a combinatorial $(2d-1)$-sphere. Then the link of $v$ in $\mathcal{U}(W_L,P')$ is the same.

Now let $v$ be a cone vertex. Then $\Lk_{ P'}(v)$ is PL-homeomorphic to $N_i$ for some $i$. 
The vertex $v$ is contained in a unique mirror of $P$, namely $\St(v,L')$, which is a cone on $\partial N_i$.  
Thus the link of $v$ in $\mathcal{U}(W_L,X)$ is the double of $N_i$ over $\partial N_i$, which is a closed combinatorial manifold.

Finally suppose $v$ is a non-cone vertex of $L'$, i.e.\ $v=v_{\gs}$ is the barycenter of some simplex $\gs \in L$. 
The link $\Lk_{P'}(v_\gs)$ decomposes as a join $(\partial \gs)' \ast \Lk_P(\gs)'$.
Since $P'$ is a pseudomanifold with isolated singularities, $\Lk_{P}(\gs)'$ is PL-homeomorphc to a simplex $\Delta$ (with the boundary $\partial \Delta$ contained in $L$), which is also PL-homeomorphic to the cone $\Cone(w,\partial \Delta)$. See Figure~\ref{fig:mirros in link}.
Let $S(v_\gs) = \{s \in S | v_\gs \in X_s\}$; note that the elements of $S(v_{\gs})$ correspond to the vertices of $\gs$. 
There is an induced mirror structure on $\Lk_{P'}(v_\gs)$ where the $s$-mirror is the star of the vertex $s$ in $(\partial \gs)' \ast \partial \Lk_{P}(\gs)'$. 
The link of $v_\gs$ in $\cu(W_L, P')$ is the corresponding basic 
construction $\cu(W_{S(v_\gs)}, \Lk_{P'}(v_\gs))$.

The PL-homeomorphism between $\Lk_{P}(\gs)'$ and $\Cone(w,\partial \Delta)$ induces a PL-homeomorphism $\Lk_{P'}(v_\gs) \rightarrow (\partial \gs)' \ast \Delta$ 
which takes the each $s$-mirror to $\St_s(\partial \gs)' \ast \partial \Delta$. We note that $(\partial \gs)' \ast \Delta = \gs' \ast \partial \Delta$, which can be seen by ``pushing'' the vertex $w$ from $\Cone(w, \partial \Delta)$ to the center of $\partial \gs$, see Figure~\ref{fig:mirros in link}.
\begin{figure}
    \centering
    \includegraphics[width=\linewidth]{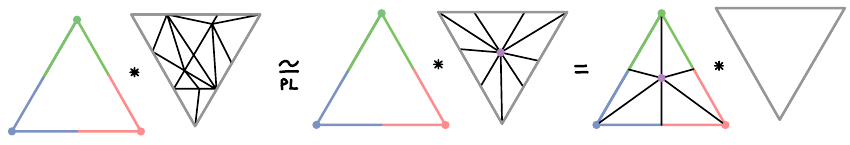}
    \caption{The link $\Lk_{P'}(v_{\sigma})$ with the induced mirror structure. Each vertex $s\in S(v_{\gs})$ is distinguished in the first picture, and the associated mirror is the join of its star in the left cycle with the boundary of the right disc. The PL-homeomorphism takes the disc $\Lk_{P'}(\sigma)$ to a cone over a cycle, whose cone vertex is distinguished in the picture.}
    \label{fig:mirros in link}
\end{figure}

Reflecting $\gs$ using the mirror structure $\St_s(\partial \gs)'$ and the right-angled Coxeter group $W_{S(v_\gs)} \cong (\mathbb{Z}/2)^{|S(v_\gs)|}$ yields a combinatorial sphere of dimension $\dim \gs$.
Thus, the link of $v_\gs$ in $\cu(W_L, P')$ is a join of two combinatorial spheres which is a combinatorial sphere.
\end{proof}

Combining Proposition~\ref{prop:reflection group trick} with Theorem \ref{t:main} gives the following:

\mainclosed*

\bibliographystyle{alpha}
\bibliography{sample}

@incollection{Bjorner,
    AUTHOR = {Bj{\"o}rner, A.},
     TITLE = {Topological methods},
 BOOKTITLE = {Handbook of combinatorics, {V}ol. 1, 2},
     PAGES = {1819--1872},
 PUBLISHER = {Elsevier Sci. B. V., Amsterdam},
      YEAR = {1995},
      ISBN = {0-444-88002-X},
   MRCLASS = {52B05 (05B35 05E25 52B40)},
  MRNUMBER = {1373690},
MRREVIEWER = {Andrew\ Vince},
}

@article{bkk,
    AUTHOR = {Bestvina, Mladen and Kapovich, Michael and Kleiner, Bruce},
     TITLE = {Van {K}ampen's embedding obstruction for discrete groups},
   JOURNAL = {Invent. Math.},
  FJOURNAL = {Inventiones Mathematicae},
    VOLUME = {150},
      YEAR = {2002},
    NUMBER = {2},
     PAGES = {219--235},
      ISSN = {0020-9910,1432-1297},
   MRCLASS = {57S30 (57M07)},
  MRNUMBER = {1933584},
MRREVIEWER = {Vyacheslav\ S.\ Krushkal},
       DOI = {10.1007/s00222-002-0246-7},
       URL = {https://doi.org/10.1007/s00222-002-0246-7},
}

@incollection{Bryant2002,
    AUTHOR = {Bryant, J. L.},
     TITLE = {Piecewise linear topology},
 BOOKTITLE = {Handbook of geometric topology},
     PAGES = {219--259},
 PUBLISHER = {North-Holland, Amsterdam},
      YEAR = {2002},
      ISBN = {0-444-82432-4},
   MRCLASS = {57Qxx (57Q15 57Q30 57Q37 57Q40)},
  MRNUMBER = {1886671},
MRREVIEWER = {R.\ J.\ Daverman},
}

@article{DranishnikovRepovs93,
    AUTHOR = {Drani\v{s}nikov, A. N. and Repov\v{s}, D.},
     TITLE = {Embeddings up to homotopy type in {E}uclidean space},
   JOURNAL = {Bull. Austral. Math. Soc.},
  FJOURNAL = {Bulletin of the Australian Mathematical Society},
    VOLUME = {47},
      YEAR = {1993},
    NUMBER = {1},
     PAGES = {145--148},
      ISSN = {0004-9727},
   MRCLASS = {57Q35 (55P10)},
  MRNUMBER = {1201430},
MRREVIEWER = {K. Horvati\'{c}},
       DOI = {10.1017/S0004972700012338},
       URL = {https://doi.org/10.1017/S0004972700012338},
}

@article{vanKampen33,
	author = {van Kampen, E.~R.},
	title = {Komplexe in euklidischen Raumen},
	journal = {Abh. Math. Sem. Univ. Hamburg},
	volume = {9}, YEAR = {1933},
	number = {1},
	pages = {72--78},

 }

@article{do01,
    author = {Davis, M. and Okun, B.},
    title = {Vanishing theorems and conjectures for the {$\ell^2$}-homology of right-angled {C}oxeter groups},
    year = {2001},
    issn = {1465-3060},
    journal = {Geom. Topol.},
    volume = {5},
    pages = {7--74},

    }

@article{dhw23,
	author = {Dani, P. and Haulmark, M. and Walsh, G.},
	title = {Right-angled {C}oxeter groups with non-planar boundary},
	journal = {Groups, Geometry, and Dynamics},
	volume = {17}, 
    YEAR = {2023},
	number = {1},
	pages = {127--155},
 }

@article{mr25,
author = {Manning, J. and Ruffoni, L.},
title = {Incubulable hyperbolic $3$-pseudomanifold groups},
year = {2026},
}

@book{Stallings65,
  title={The Embedding of Homotopy Types Into Manifolds},
  author={Stallings, J.R.},
  url={https://books.google.com/books?id=RQ4_AAAAIAAJ},
  year={1965},
  publisher={Princeton University Press}
}

@book{Davis08,
       title={The geometry and topology of {C}oxeter groups},
     author={Davis, M.W.},
      series={London Mathematical Society Monographs Series},
   publisher={Princeton University Press},
     address={Princeton, NJ},
        year={2008},
      volume={32},
        ISBN={978-0-691-13138-2; 0-691-13138-4},
      review={\MR{2360474 (2008k:20091)}},
}

@book{Zeeman63,
     
       title={Seminar on Combinatorial Topology},
     author={Zeeman, E.C.},
      series={},

publisher={Institut Des Hautes \'Etudes Scientifiques},
        year={1963},
   
}

@incollection{HogAngeloniMetzler93,
    AUTHOR = {Hog-Angeloni, Cynthia and Metzler, Wolfgang},
     TITLE = {Geometric aspects of two-dimensional complexes},
 BOOKTITLE = {Two-dimensional homotopy and combinatorial group theory},
    SERIES = {London Math. Soc. Lecture Note Ser.},
    VOLUME = {197},
     PAGES = {1--50},
 PUBLISHER = {Cambridge Univ. Press, Cambridge},
      YEAR = {1993},
      ISBN = {0-521-44700-3},
   MRCLASS = {57M20},
  MRNUMBER = {1279175},
       DOI = {10.1017/CBO9780511629358.003},
       URL = {https://doi.org/10.1017/CBO9780511629358.003},
}

@incollection{CollinsMiller99,
    AUTHOR = {Collins, Donald J. and Miller, III, Charles F.},
     TITLE = {The word problem in groups of cohomological dimension 2},
 BOOKTITLE = {Groups {S}t. {A}ndrews 1997 in {B}ath, {I}},
    SERIES = {London Math. Soc. Lecture Note Ser.},
    VOLUME = {260},
     PAGES = {211--218},
 PUBLISHER = {Cambridge Univ. Press, Cambridge},
      YEAR = {1999},
      ISBN = {0-521-65588-9},
   MRCLASS = {20F10 (20F38 57M20)},
  MRNUMBER = {1676618},
MRREVIEWER = {F.\ Levin},
}

@book{Hatcher02,
    AUTHOR = {Hatcher, Allen},
     TITLE = {Algebraic topology},
 PUBLISHER = {Cambridge University Press, Cambridge},
      YEAR = {2002},
     PAGES = {xii+544},
      ISBN = {0-521-79160-X; 0-521-79540-0},
   MRCLASS = {55-01 (55-00)},
  MRNUMBER = {1867354},
MRREVIEWER = {Donald\ W.\ Kahn},
}

\end{document}